\documentclass[12pt]{amsart}

\usepackage{amsfonts}
\usepackage{amsmath}
\usepackage{amscd}
\usepackage{amssymb}
\usepackage{amsthm}
\usepackage{bbm}
\usepackage{bm}
\usepackage{enumerate}
\usepackage{hyperref}
\usepackage{latexsym}
\usepackage{mathabx}
\usepackage{mathdots}
\usepackage{mathtools}
\usepackage[left=3cm,top=2cm,right=3cm,bottom = 2cm]{geometry}

\theoremstyle{definition}

\newtheorem{remark}{Remark}[section]
\theoremstyle{plain} 

\newtheorem{lemma}{Lemma}[section]
\newtheorem{corollary}{Corollary}[section]
\newtheorem{proposition}{Proposition}[section]

\DeclareMathOperator{\id}{id}
\DeclareMathOperator{\SL}{SL}

\newcommand{\Bers}{\textup{Bers}}
\newcommand{\Bound}{B^{1}(\Gamma,\HON)}
\newcommand{\Boundtwo}{B^{1}(\Gamma,\HOtwo)}  
\newcommand{\C}{\mathbb{C}}
\newcommand{\calS}{\mathcal{S}}
\newcommand{\Can}{\textup{Can}}
\newcommand{\Cg}{\mathfrak{C}_{g}}
\newcommand{\Cgn}{\mathfrak{C}_{g,n}}
\newcommand{\Chat}{\widehat{\C}}
\newcommand{\COH}{H^{1}(\Gamma,\HON)}
\newcommand{\COHtwo}{H^{1}(\Gamma,\HOtwo)}
\newcommand{\Con}[1]{\mathcal{C}_{#1}}

\newcommand{\D}{\mathcal{D}}
\newcommand{\delCg}{\nabla_{\hspace{-1 mm}\Cg}}
\newcommand{\delCgn}{\nabla_{\hspace{-1 mm}\Cgn}}
\newcommand{\delCgcheck}{\widecheck{\nabla}_{\hspace{-1 mm}\Cg}}
\newcommand{\delCgncheck}{\widecheck{\nabla}_{\hspace{-1 mm}\Cgn}}
\newcommand{\delCgtilde}{\widetilde{\nabla}_{\hspace{-1 mm}\Cg}}
\newcommand{\delCgntilde}{\widetilde{\nabla}_{\hspace{-1 mm}\Cgn}}
\newcommand{\delMg}{\nabla_{\hspace{-1 mm}\Mg}}
\newcommand{\delMgn}{\nabla_{\hspace{-1 mm}\Mgn}}

\newcommand{\dyb}{\partial_{\ybar}}
\newcommand{\dz}{\partial_{z}}
\newcommand{\dzb}{\partial_{\zbar}}
\newcommand{\Ei}{Z^{1}(\Gamma,\HON)}
\newcommand{\Eitwo}{Z^{1}(\Gamma,\HOtwo)}
\newcommand{\FN}{\mathcal{F}_{N}}
\newcommand{\GEM}{\Psi}
\newcommand{\gem}{\psi}
\newcommand{\GEMBers}{\Psi^{\Bers}}
\newcommand{\half}{\frac{1}{2}}
\newcommand{\HgN}{{\mathcal H}^{(g)}_{N}}
\newcommand{\Hgtwo}{{\mathcal H}^{(g)}_{2}}
\newcommand{\HH}{\mathbb{H}}
\newcommand{\HON}{{\mathcal H}^{(0)}_{1-N}}
\newcommand{\HOtwo}{{\mathcal H}^{(0)}_{-1}}

\newcommand{\I}{\mathcal{I}}
\newcommand{\im}{\textup{i}}

\newcommand{\Ip}{\mathcal{I}_{+}}
\newcommand{\J}{\mathcal{J}}
\newcommand{\Mg}{{\mathcal M}_{g}}
\newcommand{\Mgn}{{\mathcal M}_{g,n}}
\newcommand{\Omo}{\Omega_{0}(\Gamma)}
\newcommand{\PM}{R} 
\newcommand{\R}{\mathbb{R}}

\newcommand{\Schg}{\mathfrak{S}_{g}}
\newcommand{\Schgn}{\mathfrak{S}_{g,n}}
\newcommand{\Sg}{\mathcal{S}_{g}}
\newcommand{\Sgn}{\mathcal{S}_{g,n}}
\newcommand{\slLie}{\mathfrak{sl}}

\newcommand{\tpi}{2\pi \im} 

\newcommand{\ybar}{\overline{y}}
\newcommand{\yD}{{\lambda}y}
\newcommand{\Z}{\mathbb{Z}}
\newcommand{\zbar}{\overline{z}}

\title{Meromorphic Extensions of Green's Functions on a Riemann Surface}

\begin{document}
	\author{Michael P. Tuite}
	\address{School of Mathematics, Statistics and Applied Mathematics\\ 
		National University of Ireland Galway, Galway, Ireland}
	\email{michael.tuite@nuigalway.ie}
	\maketitle
	\begin{abstract}
For a Riemann surface of genus $g\ge 2$ there exists a unique Green's function $G_{N}(x,y)$ which transforms as a weight $N\ge 2$  form in $x$ and a weight $1-N$  form in $y$ and is meromorphic in $x$, with a unique simple pole at $x=y$, but is not meromorphic in $y$.  For a Schottky uniformized Riemann surface we consider meromorphic extensions of $G_{N}(x,y)$ called Green's Functions with Extended Meromorphicity or  GEM forms. GEM forms are meromorphic in both $x$ and $y$ with a unique simple pole at $x=y$,   transform as weight $N\ge 2$  forms in $x$ but as  weight $1-N$ quasiperiodic forms in $y$. We give a reformulation of the bijective Bers map and describe a choice of GEM form with an associated canonical basis of normalized holomorphic $N$-forms. We describe an  explicit differential operator constructed from $N=2$ GEM forms giving the variation with respect to moduli space parameters of a punctured Riemann surface. We also describe  a  new expression for the inverse Bers map.  
	\end{abstract}
	\section{Introduction} 	
For every Riemann surface of genus $g\ge 2$ there exists a unique Green's function $G_{N}(x,y)$ with the following properties \cite{Ma,McIT}: 
$G_{N}(x,y)$ transforms as a weight $N\ge 2$ differential form in $x$, as a weight $1-N$ differential form in $y$ and is meromorphic in $x$ with a unique simple pole at $x=y$ but is not meromorphic in $y$ where $\frac{\partial  }{\partial \ybar}  G_{N}(x,y)$ has a specific form.
% where $\frac{\partial}{\partial \overline{y} }g_{N}(x,y)$ is determined by a projection of weight $N$ differentials on to the space of holomorphic $N$ differentials. 
Using a Schottky uniformization of the Riemann surface, the definition of  $G_{N}(x,y)$ utilizes a  Poincar\'e series $\GEMBers_{N}(x,y)$  introduced by Bers \cite{Be1,Be2} as a means of constructing a Bers potential for each holomorphic weight $N$ differential form. % (generalizing the notion of a potential associated with a Beltrami differential). 
Unlike the Green's function, $\GEMBers_{N}(x,y)$ is meromorphic in both $x$ and $y$ with a unique simple pole at $x=y$ and transforms as a weight $N$ differential form in $x$ but transforms as a quasiperiodic differential form of weight $1-N$  in $y$. $\GEMBers_{N}(x,y)$ is an example  of what we call  a Green's Function with Extended Meromorphicity or GEM form, denoted by $\GEM_{N}(x,y)$. GEM forms can be viewed as a generalization of the classical differential of the third kind $\omega_{y-0}(x)$ which is meromorphic in both $x$ and $y$ and is a weight $1$ differential form in $x$ and a quasiperiodic differential form of weight $0$  in $y$ but unlike a GEM form, possesses two simple poles at $x=y$ and  $x=0$ e.g. \cite{Mu,Fa}. Here we develop some general properties of GEM forms which have important intrinsic meaning beyond their use for constructing  Bers potentials or the Green's function. Our primary motivation in studying GEM forms is that they play a pivotal role in the description of vertex operator algebras on a  Riemann surface  \cite{TW} but we believe that there may be wider interest in our results beyond that important application.

Our results in Section~2 include a new reformulation and proof of the original Bers map \cite{Be1} by describing a bijective  map between the space of Bers potentials and Eichler cocycles. We discuss a canonical choice of GEM form associated with a canonical normalized basis of holomorphic $N$-differentials which generalizes the canonical basis of holomorphic 1-differentials associated with $\omega_{y-0}(x)$. 
In Section~3 we discuss the geometrical meaning of $N=2$ GEM forms and their relation to the $3g-3$ dimensional space of holomorphic quadratic differentials. In particular, employing a canonical parameterization of Schottky space, we explicitly construct a new canonical differential operator (partially anticipated in ref.~\cite{O}) employing $\GEM_{2}(x,y)$ which gives the variation with respect to moduli space parameters of a punctured Riemann surface.  In Section~4 the general relationship between $\GEM_{N}(x,y)$ and the Green's function  $G_{N}(x,y)$ is further developed and we describe the inverse Bers map giving the Bers potential for a given choice of Eichler cocycle by means of a new expression involving  $G_{N}(x,y)$.  

\noindent \textbf{Acknowledgements.} 
I wish to thank Michael Flattery, Tom Gilroy and Michael Welby for helpful comments and suggestions.

\section{The Bers Map and GEM forms}
\subsection{The Schottky uniformization of a Riemann surface}
We briefly review the  construction of a genus $g$ Riemann surface $\Sg$ using the Schottky uniformization where we sew $g$ handles to the Riemann sphere $\calS_{0}\cong\Chat:=\C\cup{\infty}$ e.g. \cite{Fo,Bo}. Every Riemann surface can be Schottky uniformized \cite{Be3}.
Let $\{\Con{\pm a}\}$, where $a\in\Ip=\{1,\ldots, g\}$, denote a set of $2g$ non-intersecting Jordan curves  in $\C$.
Identify $z'\in \Con{-a}$ with $z\in  \Con{a}$ via the Schottky sewing cross ratio relation
\begin{align}
\label{eq:Schottkysew}
\frac{z'-W_{-a}}{z'-W_{a}}\,\frac{z-W_{a}}{z-W_{-a}}=q_{a},\quad a\in\Ip,
\end{align}
for complex $q_{a}$ with $0<|q_{a}|<1$ and some $W_{\pm a}\in\Chat$. Thus  $z'=\gamma_{a}z$  for $a\in\Ip$ for  M\"obius transformation generated by $\gamma_{a}\in \SL_{2}(\C)$  where
\begin{align}\label{eq:gammaa}
\gamma_{a}:=\sigma_{a}^{-1}
\begin{pmatrix}
q_{a}^{1/2} &0\\
0 &q_{a}^{-1/2}
\end{pmatrix}
\sigma_{a},\quad 
\sigma_{a}=(W_{-a}-W_{a})^{-1/2}\begin{pmatrix}
1 & -W_{-a}\\
1 & -W_{a}
\end{pmatrix}.
\end{align}
$\gamma_{a}$ is loxodromic  with attracting fixed point $W_{-a}$ and repelling fixed point $W_{a}$.

The marked Schottky group $\Gamma\subset \SL_{2}(\C)$  is the free discrete group of M\"obius transformations generated by $\gamma_{a}$. 
Let $\Lambda(\Gamma)$ denote the  limit set.
%\footnote{$\Lambda(\Gamma)$ is an infinite perfect set for $g\ge 2$ e.g. \cite{Fo}.}  of $\Gamma$ (the set of accumulation points of the action of $\Gamma$ on $\Chat$) and let $\Omo=\Chat-\Lambda(\Gamma)$. 
Then  $\Sg\simeq \Omo/\Gamma$, a Riemann surface of genus $g$. 
We let $\D\subset\Chat$  denote the standard connected fundamental region with oriented boundary curves $\Con{a}$. 
We further identify the standard homology cycle $\alpha_{a}$  with $ \Con{-a}$ and the cycle $\beta_{a}$ with a path connecting  $z\in  \Con{a}$ to $z'=\gamma_{a}z\in  \Con{-a}$.

Define $\gamma_{-a}=\gamma_{a}^{-1} $ so that $\gamma_{a} \Con{a}= -\Con{-a}$ for all $a\in\I=\{\pm 1, \ldots, \pm g\}$. 
Let $
w_{a}:=\gamma_{-a}.\infty$ 
 so that $
\frac{w_{a}-W_{a}}{w_{a}-W_{-a}}=q_{a}
$
which implies
\begin{align}
\label{eq:wadef}
w_{a} ={\frac {W_{{a}}-q_{{a}}W_{{-a}}}{1-q_{{a}}}},
\end{align} 
for all $a\in\I$.
We therefore have
\begin{align}
\label{eq:gamaz}
\gamma_{a}z = w_{-a}+\frac{\rho_{a}}{z-w_{a}},
\end{align}  
where $\rho_{a}=\rho_{-a}$ is 
determined from the condition $\gamma_{a}W_{a}=W_{a}$ to be
\begin{align}
\label{eq:rhoadef}
\rho_{a}=-{\frac {q_{{a}} \left( W_{{a}}-W_{{-a}} \right) ^{2}}{\left(1- q_{{a}} \right) ^{2}}}.
\end{align} 
Hence \eqref{eq:Schottkysew} can be written in the more convenient form: 
\begin{align}
\label{eq:Cansew}
(z'-w_{-a})(z-w_{a})=\rho_{a}.
\end{align}
We may  choose the Jordan curve $\Con{a}$ to be the isometric circle\footnote{The isometric circle for $\left( \begin{smallmatrix}
		a&b\\ c&d
	\end{smallmatrix} \right)\in \SL_{2}(\C)$ is given by $|cz+d|=1$ \cite{Fo}.}
of $\gamma_{a}$
of radius $|\rho_{a}|^{\half}$ centred at $w_{a}$. 
We note that the interior/exterior of the disk $\Delta_{a}=\{z:|z|\le |\rho_{a}|^{\half}\}$ is mapped by $\gamma_{a}$ to the exterior/interior of $\Delta_{-a}$ since
$|\gamma_{a}z-w_{-a}||z-w_{a}|= |\rho_{a}|$.
Furthermore, the fixed point $W_{a}\in\Delta_{a}$.
%In particular, the fixed point $W_{a}$ is an interior point of $ \Delta_{a}$.
\begin{center}
	\begin{picture}(260,100)

	\put(90,50){\circle{40}}
	\put(90,50){\makebox(0,0){$\boldsymbol{\cdot}$}}
	
	% line and r1 label
%	\put(90,50){\line(-1,1){14.1}}
	\put(90,42){\makebox(0,0){$w_{a}$}}
	
% fixed pt W_{a}	
\put(85,65){\makebox(0,0){$\boldsymbol{\cdot}$}}
\put(70,75){\makebox(0,0){$W_{a}$}}
	
	%	 C_{-a} 
	\put(90,30){\vector(1,0){0}}%arrow
	\put(90,20){\makebox(0,0){$\Con{a}$}}
	
	\put(90,50){\circle{40}}
	
	%RHS
	
	\put(160,50){\circle{40}}
	\put(160,50){\makebox(0,0){$\boldsymbol{\cdot}$}}
	% line and ra label
%	\put(160,50){\line(1,1){14.1}}
	\put(160,42){\makebox(0,0){$w_{-a}$}}

% fixed pt W_{-a}	
	\put(165,65){\makebox(0,0){$\boldsymbol{\cdot}$}}
    \put(170,75){\makebox(0,0){$W_{-a}$}}
	
	%	 C_{-a} 
	\put(160,30){\vector(1,0){0}}%arrow
	\put(160,20){\makebox(0,0){$\Con{-a}$}}
	
	% gamma a map
		\put(20,50){\qbezier(90,0)(105,-10)(120,0)}%
		\put(127,45){\vector(1,0){0}}%arrow
		\put(124,35){\makebox(0,0){$\gamma_{a}$}}
		
		% gamma inverse  a map
%	\put(20,50){\qbezier(90,0)(105,10)(120,0)}%
%	\put(123,55){\vector(-1,0){0}}%arrow
%	\put(124,62){\makebox(0,0){$\gamma_{-a}$}}	
		
			% gamma a y map
		
%		\put(105,80){\makebox(0,0){$y$}}	
%		\put(110,80){\makebox(0,0){$\boldsymbol{\cdot}$}}
%		\put(20,70){\qbezier(90,10)(115,5)(129,-10)}%
%		\put(149,60){\makebox(0,0){$\boldsymbol{\cdot}$}}
%		\put(160,60){\makebox(0,0){$\gamma_{a}y$}}
%		
%		\put(132,73){\vector(4,-1){0}}%arrow
%		\put(130,80){\makebox(0,0){$\gamma_{a}$}}
	
	\end{picture}
	
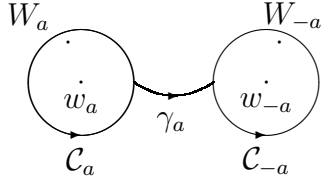
	{\small Fig.~1 Isometric Schottky Circles}
\end{center}
We define the space of Schottky parameters  $\Cg \subset \C^{3g}$ by
\begin{align}
\label{eq:Cfrakg}
\Cg:=\left\{ (w_{1},w_{-1},\rho_{1},\ldots,w_{g},w_{-g},\rho_{g}): |w_{a}-w_{b}|>|\rho_{a}|^{\frac{1}{2}}+|\rho_{b}|^{\frac{1}{2}}\; \forall \;a\neq b\right\},
\end{align}
where the condition follows from  $\Delta_{a}\cap\Delta_{b}=\emptyset$ for $a\neq b$. 
The cross ratio \eqref{eq:Schottkysew} is M\"obius  invariant with $(z,z',W_{a},q_{a})\mapsto(\gamma z,\gamma z',\gamma W_{a},q_{a})$ for $\gamma =\left(\begin{smallmatrix}A&B\\C&D\end{smallmatrix}\right)
\in\SL_{2}(\C)$ 
 giving the following
$\SL_{2}(\C)$ action on $\Cg$
\begin{align}
	\label{eq:Mobwrhoa}
	\gamma:(w_{a},\rho_{a})\mapsto & 
\left(	\frac { \left( Aw_{a}+B \right)  \left( Cw_{-a}+D \right) -\rho_{a}
		\,AC}{ \left( Cw_{a}+D \right)  \left( Cw_{-a}+D \right) -\rho_{a}\,{
			C}^{2}},
	{\frac {\rho_{a}}{ \left(  \left( Cw_{a}+D \right)  \left( Cw_{-a}+D
			\right) -\rho_{a}\,{C}^{2} \right) ^{2}}}\right).
\end{align}
We define Schottky space as $\Schg=\Cg/\SL_{2}(\C)$. $\Schg$ is a covering space for the $3g-3$ dimensional moduli space $\mathcal{M}_{g}$ of genus $g$ Riemann surfaces e.g. \cite{Be3}.

\subsection{Holomorphic  differentials ${\mathcal H}^{(g)}_{N}$}
Let ${\mathcal A} _{m,n}$ for $m,n\in\Z$ denote the vector space of smooth differentials 
of the form $\Phi(z)=\phi(z) dz^{m}d\zbar^{n}$ for local coordinate $z$ on $\Sg$ e.g.
the Poincar\'e  metric\footnote{$R$ is induced from the Poincar\'e metric $y^{-2}d\zeta d\overline{\zeta}$ on  $\HH=\{\zeta=x+ \im\,y|x,y\in\R,\, y>0\}$ by uniformizing $\Sg$ as a quotient of $\HH$ by an appropriate Fuchsian subgroup of $ \SL_{2}(\R)$.} 
$\PM(z)=\rho(z)dz d\zbar \in{\mathcal A} _{1,1}$ for real positive $\rho(z)$. $\PM$ determines the  positive definite Petersson inner product  for $\Phi,\Theta \in {\mathcal A} _{m,n}$ 
\begin{align}
\label{eq:Pprod}
\langle \Phi,\Theta \rangle   
:= \iint_{\Sg}\Phi\,
\overline{\Theta}\,\PM^{-m-n}\omega, 
\end{align} 
for real volume form%\footnote{Note that $\frac{\im}{2} dz\wedge d\zbar=dx\wedge dy$ for $z=x+iy$.} 
\begin{align}
\label{eq:dmu}
\omega(z) :=\rho(z)d^{2}z,
\end{align}
with $d^{2}z:=\frac{\im}{2} dz\wedge d\zbar$. 
Let $\mathfrak{H}^{(g)}_{m,n}$ denote the $L^2$-closure of ${\mathcal A} _{m,n}$ with respect to the Petersson product and let $\mathfrak{H}^{(g)}_{n}=\mathfrak{H}^{(g)}_{n,0}$.
Lastly, let ${\mathcal H}^{(g)}_{N}\subset \mathfrak{H}^{(g)}_{N}$ denote the space of genus $g$ \emph{holomorphic} $N$-differentials where $N$ is  referred to as the weight.   
The Riemann-Roch theorem (e.g. \cite{FK,Bo}) determines $d_{N}=\dim{\mathcal H}^{(g)}_{N}$ as follows:  
\begin{center}
	\begin{tabular}{|c|c|c|}
		\hline
		genus $g$  & weight $N$ & dimension $ d_{N}$ \\
		\hline
		$g=0$ & $N\le 0$ & $1-2N$\\
	%	\hline
	 & $N> 0$ & $0$ \\
	 \hline
	$g=1$ & $N\in\Z$ & $1$ \\
	 \hline
	$g\ge 2$ & $N< 0$ & $0$ \\
	% \hline
	 & $N=0$ & $1$ \\
	% \hline
	 & $N=1$ & $g$  \\
	 %\hline
	 & $N\ge 2$ & $(g-1)(2N-1)$ \\
	 \hline
	\end{tabular}
~\vskip 0.2 cm~
%Table~1. Values of $d_{N}=\dim{\mathcal H}^{(g)}_{N}$.
\end{center} 
In the Schottky uniformization, $\Phi(z)=\phi(z) dz^{m}d\zbar^{n}\in {\mathcal A} _{m,n}$
for $z\in\Omo$ satisfies
\begin{align*}
\Phi|_{\gamma}=\Phi,
\end{align*}
for all $\gamma\in \Gamma$ where $\Phi|_{\gamma}(z):=\phi(\gamma z) d(\gamma z)^md(\overline{\gamma z})^n$ and the Petersson product is expressed as an integral over 
the Schottky fundamental region $\D$ e.g. \cite{McI, McIT}.

\subsection{Bers potentials and GEM forms} 
We review the relationship between $\HON$ and $\HgN$ in the Schottky scheme for all $N\ge 2$  and $g\ge 2$ as developed by Bers \cite{Be1,Be2}.
 Let $\Pi_{k}$ denote the $k+1$ dimensional space of complex polynomials of degree $\le k$.
$\HON$ consists of elements
 $P(z)=p(z)dz^{1-N}$ for  $z\in \Chat$ and $p\in \Pi_{2N-2}$.
 % (with $\dim \Pi_{2N-2}=2N-1=d_{1-N}$ for $g=0$).
There is a natural M\"obius action on $\HON$ given by 
\begin{align}
\label{eq:MobPi}
P|_\gamma (z):=P (\gamma z),
\end{align}
for $\gamma\in\SL_{2}(\C)$ with $P|_{\gamma\lambda}=P|_{\gamma}|_{\lambda}$ for all $\gamma,\lambda\in\SL_{2}(\C)$.

Let $\Gamma$ be a Schottky group for a Riemann surface of genus $g\ge 2$.
Let $\Ei$ denote the vector space of Eichler 1-cocycles\footnote{We define  a 1-cocycle  as a mapping to $\HON$ rather than to $\Pi_{2N-2}$ as in refs.~\cite{Be1,G}.} for $\Gamma$ given by mappings of the form $\Xi:\Gamma\rightarrow \HON$ such that for all $\gamma,\lambda\in\Gamma$
\begin{align}
\label{eq:cocycle}
\Xi[\gamma\lambda]=\Xi[\gamma]|_{\lambda}+\Xi[\lambda],
\end{align} 
for M\"obius action \eqref{eq:MobPi}. We note that \eqref{eq:cocycle} implies $\Xi[\id]=0$ so that 
$\Xi[\gamma]=-\Xi[\gamma^{-1}]|_{\gamma}$.

Let $\Bound\subset \Ei$ denote the  space of coboundaries  
$\Xi_{P}:\Gamma\rightarrow \HON$ for $P\in \HON$  given by
\begin{align}
\label{eq:cobound}
\Xi_{P}[\gamma]:=P|_{\gamma}-P.
\end{align}
$\Xi_{P}$ is a 1-cocycle since $P|_{\gamma}|_{\lambda}-P= \left(P|_{\gamma}-P\right)|_{\lambda}+P|_{\lambda}-P$. It is easy to show that \cite{Be1}  
\begin{lemma}\label{lem:dimcobound}
$\Bound\simeq \HON$ as vector spaces.
\end{lemma} 
%\begin{proof}
%For a cocycle $\Xi$ on $\Gamma$ we may define a conjugate cocycle 
%$\Xi ^{\alpha}$ on $\Gamma^{\alpha}=\alpha^{-1}\Gamma\alpha $ for  $\alpha\in\SL_{2}(\C)$ given by
%\[
%\Xi ^{\alpha}[\alpha^{-1}\gamma\alpha]=\Xi [\gamma]|_{\alpha}.
%\] 
%In particular, for two Schottky group generators $\gamma_{1},\gamma_{2}$ (recall $g\ge 2$) we may choose $\alpha$ such that 
%\[
%\left(\alpha^{-1}\gamma_{1}\alpha\right)z=q_1 z,\qquad \left(\alpha^{-1}\gamma_{2}\alpha\right)z=\frac{az+b}{cz+d},
%\]
%for $0<|q_1|<1$ and with $bc\neq  0$ and $\left(\tr \gamma_{2}\right)^{2}=(a+d)^{2}\notin[0,4]$ (since $\gamma_{2}$ is loxodromic). Suppose there exists $P=p(z)dz^{1-N}\in\HON$ with trivial coboundary $\Xi_{P}^{\alpha}=0$. Then 
%\[
%\Xi_{P}^{\alpha}[\gamma_{1}]=p(q_1z)q_1^{1-N}-p(z)=0,
%\]
%so that $p(z)=Az^{N-1}$ for some constant $A$. Furthermore
% \[
%\Xi_{P}^{\alpha}[\gamma_{2}]=A\left((az+b)^{N-1}(cz+d)^{N-1}-z^{N-1}\right)=0,
% \]
%so that $A=0$ (otherwise $ac=bd=0$ and $a+d\neq 0$ imply $b=c=0$). Hence $\Xi_{P}^\alpha=0$ iff $P=0$ and the result follows. 
%\end{proof}

Let $\COH:=\Ei/\Bound$ be the cohomology space of Eichler cocycles modulo  coboundaries.
\begin{lemma}{\cite{Be1}}\label{lem:dimH}
$\dim \COH =(g-1)(2N-1)$. 
\end{lemma} 
\begin{proof}
Since $\Gamma$ is freely generated by $\{\gamma_{a}\}$ for $a\in\Ip$, a cocycle $\Xi$ is determined by its evaluation on $\gamma_{a}$. Thus 
\[
\dim\Ei=g\dim\HON=g(2N-1). 
\]
The result follows  on applying Lemma~\ref{lem:dimcobound}.
\end{proof}
Note that $\dim\COH=\dim\HgN$ for $N\ge 2$ and $g\ge 2$. The Bers map discussed below describes  a bijection between these spaces. 
\medskip

$F(y)=f(y)dy^{1-N}$ for $f(y)$ continuous for $y\in\Omo$  is called a Bers potential  for a holomorphic $N$-form $ \Phi =\phi (y)dy^{N}\in \HgN$ provided $f(y)$ satisfies\footnote{We include a factor of $\frac{1}{\pi}$ in comparison to \cite{Be1, McIT} for later convenience.}
\begin{align}
\label{eq:dF}
&\frac{1}{\pi}\dyb f=\overline{\phi (y)}\,\rho(y)^{1-N},
\\
&\lim_{y\rightarrow 0}\left\vert y ^{2N-2}f\left(y^{-1}\right)\right\vert <\infty, 
\label{eq:Flim}
\end{align} 
where $\dyb=\frac{\partial}{\partial \ybar}$. 
\eqref{eq:dF} can also be written in the following  coordinate-free way 
\begin{align}
\label{eq:dPhiF}
\frac{1}{\tpi}d \left(F\Theta  \right)=\Theta \,\overline{\Phi}\,\PM^{-N}\omega,
\end{align} 
for all holomorphic  $\Theta \in\HgN$ with exterior derivative $d\left(h(y)dy\right)=-\dyb h \,dy\wedge d\ybar$ and volume form $\omega$ of \eqref{eq:dmu}.
\eqref{eq:Flim} ensures that $F(y)$ is defined at the point at infinity. Let $\mathcal{F}_{N}$ denote the vector space of  Bers potentials.
It is straightforward to see that
\begin{lemma}{\cite{Be1}}\label{lem:ZeroPot}
$F\in\FN$ is a  Bers potential for $\Phi =0$ iff  $F\in\HON$.
\end{lemma}
A Bers potential $F$ is quasiperiodic under the action of the Schottky group.  Define
\begin{align}
\Xi_{F}[\gamma]:=F|_{\gamma}-F,\quad \gamma\in\Gamma,
\label{eq:xiF}
\end{align}
where
$F|_\gamma (y)=F(\gamma y)$. We then find
\begin{lemma}\label{lem:xiF_cocycle}
$\Xi_{F}$ is an Eichler 1-cocycle for each  $F\in\FN$. 
\end{lemma}
\begin{proof}
\eqref{eq:dPhiF} implies $
d\left((F\Theta)|_{\gamma}-F\Theta\right)=0$ 
for all $\Theta\in\HgN$
so that $\Xi_{F}[\gamma]\in\HON$. 
$\Xi_{F}$ is a 1-cocycle since 
$F|_{\gamma\lambda}-F=\left(F|_{\gamma}-F\right)|_{\lambda}+F|_{\lambda}-F$.
\end{proof}
 
The existence of a Bers potential for $\Phi \in \HgN$ for $N\ge 2$ and $g\ge 2$ follows from consideration of the following Poincar\'e series   defined for all $x,y\in \Omo$  by \cite{Be1,Be2} 
\begin{align}
\label{eq:GEMBers}
\GEMBers_{N}(x,y):=\sum_{\gamma\in\Gamma} 
\frac{1}{\gamma x-y}\prod_{j=1}^{2N-1}\frac{y-A_j}{\gamma x-A_j}d(\gamma x)^N dy^{1-N}.
\end{align}
Here $A_1,\ldots, A_{2N-1}$ are any distinct   elements of the limit set $\Lambda(\Gamma)$. 
$\GEMBers_{N}(x,y)$ is meromorphic  in $x,y\in \Omo$ with a simple pole  of residue one at $y=\gamma x$ for all $\gamma\in\Gamma$  \cite{Be1,Be2,G,McI,McIT}. Furthermore, $\GEMBers_{N}(x,y)$ is a bidifferential $(N,1-N)$-quasiform with respect to the Schottky group as follows.
By construction, $\GEMBers_{N}(x,y)$ is an $N$-differential in $x$ so that for all $\gamma\in\Gamma$
\begin{align}
\GEMBers_{N}(\gamma x,y) &= \GEMBers_{N}(x,y).
\label{eq:GEMBersx}
\end{align}
$\GEMBers_{N}(x,y)$ is a quasiperiodic $1-N$ form  in $y$ where  
\begin{align}\label{eq:GEMBersgamy}
\GEMBers_{N}( x,\gamma y) - \GEMBers_{N}(x,y)&=\chi^{\Bers}[\gamma](x,y),
\end{align}
for $\gamma\in\Gamma$ where $\chi^{\Bers}[\gamma](x,y)$ is holomorphic for $x,y\in\Omo$. %From the above discussion we find that $\chi^{\Bers}[\gamma](x,y)$ is a holomorphic  $N$-form in $x$ and a 1-cocycle in $y$.

For a given $N$-differential $ \Phi  \in \HgN$ and $y\in\Omo$ we define 
\begin{align*}
F^{\Bers}(y):=-\iint_{\D}\GEMBers_{N}(\cdot,y)\,
\overline{\Phi}\,\PM^{-N}\omega
=-\langle \GEMBers_{N}(\cdot,y), \Phi \rangle.
\end{align*}
$F^{\Bers}(y)=f^{\Bers}(y)dy^{1-N}$ satisfies \eqref{eq:Flim} with  $f^{\Bers}(y)$ continuous on $\Omo$ \cite{Be1}. Since 
\begin{align}
\label{eq:fdelta}
f(y)=-\frac{1}{\pi}\dyb\iint_{R} \frac{f(z)}{z-y}\,d^{2}z,\quad y\in R,
\end{align}
for any complex function $f(z)$ on an open region $R\subset\C$ (e.g. \cite{GR}), we find that  
 $F^{\Bers}(z)$ satisfies \eqref{eq:dF}. Thus $F^{\Bers}(y)$ is a Bers potential for $\Phi  $  
with cocycle $\Xi^{\Bers}$ from \eqref{eq:xiF} and we find:
\begin{proposition}{\cite{Be1}}
There exists a Bers potential  for each $ \Phi  \in \HgN$.
\end{proposition}

By Lemma~\ref{lem:ZeroPot} the most general Bers potential  for $\Phi\in\HgN$ is of the form 
$
F=F^{\Bers}+P
$
for some $P\in\HON$ with  1-cocycle $\Xi=\Xi^{\Bers}+\Xi_{P}$. 
Let $\{\Phi _{r}\}_{r=1}^{d_{N}}$  be a  $\HgN$-basis of dimension $d_{N}=(g-1)(2N-1)$ where  $\Phi _{r}$ has potential $F_{r}(y)=F^{\Bers}_{r}(y)+P_{r}(y)$ with  cocycle $\Xi_{r}=\Xi^{\Bers}_{r}+\Xi_{P_{r}}$ for some choice of $P_{r}\in\HON$.
Let $\{\Phi ^{\vee}_{r}\}_{r=1}^{d_{N}}$ be the Petersson dual basis. 
We may define the following meromorphic bidifferential $(N,1-N)$-quasiform
\begin{align}
\label{eq:GEMdef}
\GEM_{N}(x,y):=\GEMBers_{N}(x,y)-\sum_{r=1}^{d_{N}}\Phi ^{\vee}_{r}(x)P_{r}(y),
\end{align}
where  the Bers potential for $\Phi _{r}$ is given by
\begin{align} 
{F}_{r} (y)=-\langle \GEM_{N}(\cdot,y), \Phi _{r} \rangle.
\label{eq:Bers_pot}
\end{align}
 $\GEM_{N}(x,y)$ is an $N$-form in $x$ and a quasiperiodic $1-N$ form  in $y$ with 
\begin{align}\label{eq:GEMgamy}
\GEM_{N}( x,\gamma y) - \GEM_{N}(x,y)&={\chi}[\gamma](x,y),
\end{align}
where ${\chi}[\gamma](x,y)$ is holomorphic for all $x,y\in \Omo$. 
Hence \eqref{eq:GEMBersx} and \eqref{eq:Bers_pot} imply
	\begin{align}
	\label{eq:chigam}
	{\chi}[\gamma](x,y)=&-\sum_{r=1}^{d_{N}}\Phi ^{\vee}_{r}(x){\Xi}_{r}[\gamma](y).
	\end{align}
We refer to $\GEM_{N}$ of \eqref{eq:GEMdef} as a Green's function with Extended Meromorphicity or a GEM form  for reasons explained in \S\ref{sec:Green's}. The space of  GEM forms is of dimension $(2N-1)d_N=(g-1)(2N-1)^{2}$ from \eqref{eq:GEMdef}. 

 \subsection{The Bers map}  
The following commutative diagram  summarizes the various maps introduced in the previous  section
 \begin{align}
 \label{eq:maps}
 \begin{CD}
 \mathcal{F}_{N}  @>\alpha>> \Ei
 \\
 @V\epsilon VV @V\delta VV  
 \\
 \HgN  @>\beta >>  \COH
 \end{CD}
 \end{align}
 where $\alpha$ is the linear map determined by \eqref{eq:xiF}, $\epsilon $ is the complex anti-linear map determined by  \eqref{eq:dF} with pre-image determined by \eqref{eq:Bers_pot}  and $\delta$ is  the coboundary quotient map with $\Bound=\ker \delta $. Then 
 Lemmas~\ref{lem:dimcobound} and \ref{lem:ZeroPot} are equivalent to
 \begin{align}\label{eq:kernels}
 \ker \epsilon  =\HON,\quad \ker \delta=\alpha (\ker \epsilon  ).
 \end{align}
The complex antilinear mapping $\beta $ is known as the Bers map. 
  We have the following fundamental result (which is a reformulation of Bers' classic result \cite{Be1})  
%\begin{proposition}\label{prop:BersMap}
%The Bers map 
%$	\beta :\HgN\rightarrow \COH$ is 
%bijective.
%\end{proposition}	
%We prove Proposition~\ref{prop:BersMap} using the :
\begin{proposition}\label{prop:BersMapalpha}
$\mathcal{F}_{N}\simeq \Ei$ as vector spaces.
\end{proposition} 
\begin{proof}
We first show that $\alpha:\mathcal{F}_{N} \rightarrow \Ei$ is injective. Let $F\in\mathcal{F}_{N}$ be a potential for $\Phi\in\HgN$ such that $\Xi_{F}=\alpha(F)=0$. Thus $F\vert_{\gamma}=F$ for all $\gamma\in\Gamma$ implying $F\in\mathfrak{H}^{(g)}_{1-N}$.
Hence \eqref{eq:dPhiF} implies that for all  $\Theta \in\HgN$ we have 
\begin{align}
\langle \Theta,\Phi\rangle & = \frac{1}{\tpi}\iint_{\D}d \left(F\Theta  \right)
=-\sum_{a\in\I}\frac{1}{\tpi}\oint_{\mathcal{C}_a}\Theta F
\notag
\\
&=-\frac{1}{\tpi}\sum_{a=1}^{g}\oint_{\Con{a}}\Theta \left(F-F|_{\gamma_{a}}\right)=0,
\label{eq:zeroStokes}
\end{align}
by Stokes' theorem on the fundamental domain $\mathcal D$  formed from $\Chat$ by excising $2g$ discs with oriented boundary curves $\mathcal{C}_a$  and that $\Con{-a}=-\gamma_{a} \Con{a}$. 
Since $\langle \, ,\,\rangle $ is invertible we have $\Phi=0$ which implies $F$ is holomorphic
by \eqref{eq:dPhiF}. Hence $F\in\mathcal{H}^{(g)}_{1-N}$ so that $F=0$ by the Riemann-Roch theorem.
Lemma~\ref{lem:ZeroPot} implies $\dim\mathcal{F}_{N}=g(2N-1)=\dim\Ei$ so that $\alpha$ is also surjective and thus bijective.
\end{proof}

\begin{remark}\label{rem:Bersmap}
Proposition~\ref{prop:BersMapalpha} together with  \eqref{eq:kernels} imply that the Bers map $\beta:  \HgN \rightarrow \COH$ is also bijective.
The statement and proof of  Proposition~\ref{prop:BersMapalpha} can be adapted to any Kleinian group  $\Gamma$  since $\D$ consists of a finite set of disconnected components \cite{Be1} and hence $\langle \Theta,\Phi\rangle=0$, as in \eqref{eq:zeroStokes}, so  that  $\alpha$ and  $\beta$ are injective.
\end{remark}

We note the following useful identity \cite{Be1,McIT}  
\begin{proposition}\label{prop:ThetaPhi}
Let $\Xi \in\Ei$  with $\Phi=\left( \epsilon \circ\alpha^{-1} \right)( \Xi)\in\HgN$. Then 
\begin{align}
\langle \Theta,\Phi \rangle =  \frac{1}{\tpi}\sum_{a=1}^{g}\oint_{\Con{a}}\Theta\,\Xi[\gamma_{a}],
\label{eq:ThetaPhi}
\end{align}
for all $\Theta \in\HgN$.
% where the summation is over \textbf{either} $\Ip$ or $\Imin$.
\end{proposition}

\begin{proof}
Let $F=\alpha^{-1}(\Xi)$ be the Bers potential for $\Phi $ with cocycle $\Xi$. Then
	\begin{align*}\notag
	\iint_{\D}\Theta \, \overline{\Phi}\,\PM^{-N}\omega&
	=\frac{1}{\tpi}\iint_{\D}d\left( \Theta \,
	F \right)
	=\frac{1}{\tpi}\sum_{a=1}^{g}\oint_{\Con{a}}\Theta \,\Xi[\gamma_{a}],
	\end{align*}
much as in \eqref{eq:zeroStokes}. 
	%	
	%	
	%	The result  \eqref{eq:intThetaXi} with sum indexed by $a\in\Imin$ follows from \eqref{eq:ThetaXiplus} and using  $\Xi[\gamma_{a}]=-\Xi[\gamma_{-a}]|_{\gamma_{a}}$ from \eqref{eq:Xiinverse}.
\end{proof}
\begin{remark}\label{rem:Xirep}
   $\sum_{a=1}^{g}\oint_{\Con{a}}\Theta\,\Xi_{P}[\gamma_{a}]=0$ for any coboundary cocycle $\Xi_{P}$.
\end{remark}
\begin{corollary}\label{cor:PhiCont}
	Let $\{\Phi _{s}\}_{s=1}^{d_{N}}$  be a $\HgN$-basis and let $\{\Phi ^{\vee}_{r}\}_{r=1}^{d_{N}}$ be the Petersson dual basis. For  $ r,s=1,\ldots,d_{N}$ we have\footnote{We note that there  appears to be a sign error  in  (4.1) of \cite{McIT}.}
	\begin{align}
	\frac{1}{\tpi}\sum_{a=1}^{g}\oint_{\Con{a}}\Phi ^{\vee}_{r}\,\Xi_{s}[\gamma_{a}]=\delta_{rs},
	\label{eq:intPhiXi}
	\end{align}
	where $\Xi_{s}=\beta \left(\Phi _{s}\right)$ is any cocycle representative associated with $\Phi_{s}$.
\end{corollary}
%We may now give a proof of  the Bers bijection Theorem~\ref{prop:BersMap}. 
%\begin{proof}
%Let $\{\Phi _{s}\}$ be a $\HgN$-basis with  dual basis $\{\Phi^{\vee}_{r}\}$.	For $\Phi\in\HgN$ with $\Xi=\beta \Phi$ we have $\Phi =\sum_{r=1}^{d_{N}}c_{r} \Phi _{r}$ where from \eqref{eq:ThetaPhi}
%\begin{align}
%\label{eq:Phicr}
%\overline{c}_{r}%=\langle \Phi^{\vee}_{r}, \Phi \rangle
% =\frac{1}{\tpi}\sum_{a=1}^{g}\oint_{\Con{a}}\Phi ^{\vee}_{r}\,\Xi[\gamma_{a}].
%\end{align}
%We first prove that $\beta $ is injective.
%Consider $\Phi \in\HgN$ such that $\Xi=\beta \Phi $ is a coboundary i.e. $\Xi=\Xi_{P}$
%for some $P\in\HON$ of \eqref{eq:cobound}. Then Remark~\ref{rem:Xirep} implies that 
%$\overline{c}_r=0$ so that $\Phi =0$. Hence $\beta $ is injective.   Lemma~\ref{lem:dimH} implies $\beta $ is  surjective and the result follows.
%\end{proof}

Let $\{\Phi _{s}\}_{s=1}^{d_{N}}$  be a $\HgN$-basis with  potentials $\{F_{s}\}_{s=1}^{d_{N}}$ and cocycles $\{\Xi_{s}\}_{s=1}^{d_{N}}$ i.e. $\Xi_{s}$ is a \emph{particular}  cocycle representative associated with $\Phi_{s}$. Recall there exists a  corresponding   GEM form $\GEM_{N}$ of \eqref{eq:GEMdef} determining the potentials $\{F_{s}\} _{s=1}^{d_{N}}$ in \eqref{eq:Bers_pot}.
We then find
\begin{proposition}\label{prop:GEMCont}
$\GEM_{N} $ and the cocycles $\{\Xi_{s}\} _{s=1}^{d_{N}}$ obey
		\begin{align}
		\label{eq:intpsixi}
	\frac{1}{\tpi}	\sum_{a=1}^{g}\oint_{\Con{a}}\GEM_{N}(\cdot,y)\Xi_{s}[\gamma_{a}](\cdot)
=\Xi_{s}[\lambda](y),
	\end{align}
where  $y\in\Omo$ and $\lambda$ is the unique Schottky group element such that $\lambda y\in\D$.
\end{proposition}
\begin{proof}
$\GEM_{N}(x,y)$ has a unique simple pole at 
$x=\yD$ for $x\in \D$ using \eqref{eq:GEMBers}. 
Let $\D_{\varepsilon}=\D-\Delta_{\varepsilon}$ where $\Delta_{\varepsilon}$ is a disc of arbitrarily small radius $\varepsilon$ centered at $\yD$ with oriented boundary $C_{\varepsilon}$. Parameterizing $\Delta_{\varepsilon}$ by $x=\yD+re^{\im\theta}$ for $0\le r\le \varepsilon$ and $0\le\theta\le 2\pi$ we find $\GEM_{N}(x,y)\omega =\left(\rho(\yD)+O(\varepsilon)\right)dr\wedge d\theta$ for $x\in\Delta_{\varepsilon}$ so that \eqref{eq:Bers_pot} implies the Bers potential $F_{s}=\alpha^{-1}\left(\Xi_{s}\right) $  for $\Phi_{s}$  is given by
\begin{align*}
{F_{s}}(y)&=
-\lim_{\varepsilon\rightarrow 0}
\iint_{\D_{\varepsilon}}\GEM_{N}(\cdot,y)\,
\overline{\Phi}_{s}\,\PM^{-N}\omega.
\end{align*}
Similarly to \eqref{eq:dPhiF}, we find  that for fixed $y$
\begin{align*}
\frac{1}{\tpi}d\left(\GEM_{N}(\cdot,y)F_{s} \right)
=\GEM_{N}(\cdot,y)\overline{\Phi}_{s}\,\PM^{-N}\omega,
\end{align*} 
on $\D-\{\lambda y\}\supset \D_{\varepsilon}$.
Then Stokes' theorem implies
\begin{align*}
-\iint_{\D_{\varepsilon}}\GEM_{N}(\cdot,y)\,
\overline{\Phi}_{s}\,\PM^{-N}\omega
%=&
%-\frac{1}{\tpi}\iint_{\D_{\varepsilon}}
%d\left(\GEM_{N}(\cdot,y)F_{s} \right)
%\\
=&
\frac{1}{\tpi}\sum_{a\in\I}\oint_{\Con{a}}\GEM_{N}(\cdot,y)F_{s} 
+\frac{1}{\tpi}\oint_{\mathcal{C}_{\varepsilon}}\GEM_{N}(\cdot,y)F_{s} 
\\
=&
-\frac{1}{\tpi}\sum_{a=1}^{g}\oint_{\Con{a}}\GEM_{N}(\cdot,y)\Xi_{s}[\gamma_{a}]
+\frac{1}{\tpi}\oint_{\mathcal{C}_{\varepsilon}}\GEM_{N}(\cdot,y)F_{s} .
\end{align*}
Let $x=\yD+\varepsilon e^{\im\theta}$ on  $C_{\varepsilon}$ so that  $\GEM_{N}(x,y)F_{s} (x) =\im \left(F_{s}(\yD)+O(\varepsilon)\right)  d\theta$ implying
\begin{align*}
\lim_{\varepsilon\rightarrow 0}\frac{1}{\tpi}\oint_{\mathcal{C}_{\varepsilon}}\GEM_{N}(\cdot,y)F_{s} =
F_{s} (\yD).
\end{align*}
Combining these identities we find
\begin{align*}
\frac{1}{\tpi}\sum_{a=1}^{g}\oint_{\Con{a}}\GEM_{N}(\cdot,y)\Xi_{s}[\gamma_{a}]=
F_{s} (\yD)- F_{s}(y)=\Xi_{s}[\lambda](y).
\end{align*}
\end{proof}
\begin{remark}
	\label{rem:KNhatF}
Since $\Xi_{s}[\id]=0$, Proposition~\ref{prop:GEMCont} implies that  for   $y\in\D$  we have
\begin{align}\label{eq:GEMCont0}
\sum_{a=1}^{g}\oint_{\Con{a}}\GEM_{N}(\cdot,y)\Xi_{s}[\gamma_{a}]=0.
\end{align}
Conversely, \eqref{eq:GEMCont0} implies \eqref{eq:intpsixi} using   \eqref{eq:GEMgamy},  \eqref{eq:chigam} and \eqref{eq:intPhiXi}.
\end{remark}

\subsection{A canonically normalized basis for $\HgN$}\label{sect:normalized basis}
We define a canonical $\Ei$ homology basis $\{\Xi_{ak}\}$ for $a\in \Ip$ and $k=0,\ldots,2N-2$ determined by their evaluation on the marked Schottky group generating set,	$\{\gamma_{b}\}$ for $b\in\Ip$, as follows
\begin{align}
\label{eq:xinorm}
\Xi_{ak}[\gamma_{b}](z):=\delta_{ab}z_{a}^{k}\, dz^{1-N},
\end{align}
with  $z_{a}=z-w_a  $ for Schottky parameter $w_{a}$. 
For  any given $P\in\HON$ with coboundary $\Xi_{P}$  we may write 
\[
\Xi_{P}[\gamma_{a}](z)=\sum_{k=0}^{2N-2}p_{ak} z_{a}^k \, dz^{1-N},\quad  a \in\Ip, 
\] 
for  some complex coefficients $p_{ak}$ determined by $P$.  Hence it follows that  
\begin{align}
\label{eq:cocyclerels}
\Xi_{P}=\sum_{a=1}^{g}\sum_{k=0}^{2N-2}p_{ak}\Xi_{ak},
\end{align}
since they coincide on each generator $\gamma_{a}$. Thus there are $2N-1$ linear relations modulo coboundaries on the set $\{\Xi_{ak}\}$ (cf. Lemma~\ref{lem:dimH}). 
Let $\Phi_{ak}=\left( \epsilon \circ\alpha^{-1} \right) \left(\Xi_{ak}\right)\in\HgN$. 
Since $\left( \epsilon \circ\alpha^{-1} \right) \left(\Xi_{P}\right)=0$ there are $2N-1$ linear relations on the set $\{\Phi_{ak}\}$ given by
\begin{align}
\label{eq:Phidep}
 \sum_{a=1}^{g}\sum_{k=0}^{2N-2}\overline{p_{ak}}\Phi_{ak}=0.
\end{align}
Let $\{\Phi_{ak}\}_{\J}$ denote any $\HgN$-basis indexed by $ (a,k)\in \J$ where $\J$ is  a set, of cardinality $d_{N}$, of distinct $(a,k)$ values with $a=1,\ldots, g$ and $k=0,\ldots, 2N-2$
%We let $\J'$ denote the complement set of cardinality $2N-1$
 i.e. the cocycles $\{\Xi_{ak}\}_{\J}$ are independent modulo coboundaries.
We let  $\{\Phi^{\vee}_{ak}\}_{\J}$ denote the Petersson dual basis and let 
$\GEM_{N}^{{\Can}}$ be the corresponding  GEM form  \eqref{eq:GEMdef}.
Then   \eqref{eq:chigam},  Corollary~\ref{cor:PhiCont} and Remark~\ref{rem:KNhatF} imply
\begin{proposition}\label{prop:Can_norm}
A $\HgN$-basis $\{\Phi _{ak}\}_{\J}$ with canonical cocycles $\{\Xi_{ak}\}_{\J}$, 
Petersson dual basis $\{\Phi^{\vee}_{bl}\}_{\J}$ and GEM form $\GEM_{N}^{{\Can}}$ obeys
\begin{align}
\oint_{\Con{a}}\GEM_{N}^{{\Can}}(x,y)\Xi_{ak}[\gamma_{a}](x)
&=0,
\label{eq:CanPsi_norm}
\\
\GEM_{N}^{{\Can}}(x,y)-\GEM_{N}^{{\Can}}( x,\gamma_{a} y) 
&=
\sum_{(a,k)\in \J}\Phi ^{\vee}_{ak}(x)\Xi_{ak}[\gamma_{a}](y),
\label{eq:CanPsiQPeriod}
\\
\frac{1}{\tpi}\oint_{\mathcal{C}_a}\Phi^{\vee}_{bl}(x) \Xi_{ak}[\gamma_{a}](x)&=
%\frac{1}{\tpi}\oint_{\mathcal{C}_a}\Phi^{\vee}_{bl}(x) (x-w_{a})^k\,dx^{1-N}=
\delta_{ab}\delta_{kl},
\label{eq:CanPhi_norm}
\end{align} 
for all $y\in\D$ and $(a,k),(b,l)\in\J$.
\end{proposition}

\begin{remark}\label{rem:basis}
Proposition~\ref{prop:Can_norm} is a natural generalization of the properties of the classical differential of the third kind 
\begin{align*}
\omega_{y-0}(x):=\sum_{\gamma\in\Gamma} 
\left(\frac{1}{\gamma x-y}-\frac{1}{\gamma x}\right)d(\gamma x),
\end{align*}
which is a 1-differential is $x$ and a 0-differential in $y$ where $\omega_{y-0}(x)-\omega_{\gamma_{a}y-0}(x)=\nu_{a}(x)$  
for holomorphic 1-differential $\nu_{a}$ normalized by $\oint_{\Con{a}}\nu_{b}=\tpi\delta_{ab}$. 
However, $\GEM_{N}(x,y)$ has a unique simple pole at $x=y$ whereas $\omega_{y-0}(x)$ has an additional simple pole  at $x=0$ which is why we excluded the case $N=1$ in the above discussion.
\end{remark}

\section{Variation of the Riemann Surface Moduli} 
We consider Bers potentials, holomorphic  differentials and GEM functions for $N=2$ where the structures discussed in \S2 have a particular geometric significance.  
The existence of a complex structure on a Riemann surface is equivalent to that of a Riemannian metric with line element
\begin{align*}
ds^2  \sim| dz +\mu(z,\zbar)\, d\zbar|^{2},
\end{align*} 
for local coordinates $z,\zbar$ where $|\mu|<1$ and  $B(z,\zbar):=\mu(z,\zbar)dz^{-1} d\zbar\in {\mathcal A} _{-1,1}$  e.g. \cite{GL}.
 $B(z,\zbar)$ is  called a Beltrami differential. The metric can be transformed locally to coordinates $w,\overline{w}$  where $ds^{2} \sim|dw| ^{2}$ provided $w(z,\zbar)$ satisfies the Beltrami equation 
\begin{align}
\label{eq:beltrami}
\partial_{\zbar}w=\mu\partial_{z}w.
\end{align}
The mapping $z\rightarrow w(z)$ is called a quasiconformal map.  
%All genus $g$ compact Riemann surfaces are quasi-conformal to each other.
Define $\phi =\overline{\mu}\rho$, for the Poincar\'e metric $\rho$, so that  $\Phi(z)=\phi(z) dz^2\in {\mathcal A} _{2,0}$. $B(z,\zbar)$   is called an harmonic Beltrami differential  when $\partial_{\zbar}\phi=0$ so that $\Phi(z)\in \Hgtwo$, the space of holomorphic quadratic differentials. 
There is a 1-1 map between the infinitesimal variations of $\Mg$  and the space of harmonic Beltrami differentials  i.e. a bijective antilinear map between the moduli tangent space $T(\Mg)$ and $\Hgtwo$ e.g. \cite{A}.

We may explicitly realise  these ideas in the Schottky uniformization as follows. Consider a small variation in a Schottky parameter $m\rightarrow m+\varepsilon_{m}$ with corresponding  quasiconformal map given by $z\rightarrow w(z,\zbar,\varepsilon_{m})$ where
\begin{align}\label{eq:waz}
w=z+\frac{\varepsilon_{m}}{\pi}f_{m}+O(\varepsilon_{m}^2),
\end{align}
for some\footnote{The factor of $\frac{1}{\pi}$ is introduced to comply with our Bers potential definition \eqref{eq:dF}.} $f_{m}(z,\zbar)$. 
\eqref{eq:beltrami} implies that $\mu=\varepsilon_{m}\mu_{m}+O(\varepsilon_{m}^2)$ where
\[
\mu_{m}=\frac{1}{\pi}\partial_{\zbar}f_{m} .
\]
Thus for an harmonic Beltrami differential 
$F_{m}=f_{m}(z)dz^{-1}\in \mathcal{F}_{2}$ is a Bers potential for $\Phi_{m}=\overline{\mu}_{m}\rho dz^{2}\in\Hgtwo$ from \eqref{eq:dF} with $N=2$. 
The deformed Riemann surface is uniformized with a Schottky group $\Gamma_{\varepsilon_{m}}$  where for each  $\gamma\in\Gamma$ we define $\gamma_{\varepsilon_{m}}\in\Gamma_{\varepsilon_{m}}$ via the compatibility condition: $
\gamma_{\varepsilon_{m}}w(z)=w(\gamma z)$. 
But $\gamma_{\varepsilon_{m}}z=\gamma z+\varepsilon_{m} \partial_{m}(\gamma z)+O(\varepsilon_{m}^2)$ and using \eqref{eq:waz} we find \cite{EO,Ro,P} 
\begin{align}\label{eq:delmgamz}
\partial_{m}(\gamma z)=&\frac{1}{\pi}\left(f_{m}(\gamma z)- f_{m}(z)(\gamma z)'  \right)
=\frac{1}{\pi}\Xi_{m}[\gamma](z)d(\gamma z),
\end{align}
where $\Xi_{m}$ denotes  the 1-cocycle for potential $F_{m}$. 
Define the following  $T(\Cg)$ basis:
\begin{align}
\label{eq:delael}
\partial_{a,0}:=\partial_{w_{a}},\quad
\partial_{a,1}:=\rho_{a} \partial_{\rho_{a}},\quad
\partial_{a,2}:=\rho_{a} \partial_{w_{-a}},
\end{align}
for $a\in\Ip$ and $\ell=0,1,2$.
For each generator $\gamma_{b}\in\Gamma $ of \eqref{eq:gamaz} we have
\begin{align*}
\partial_{a,k}(\gamma_{b}z)=\rho_{a}(z-w_{a})^{k-2}\delta_{ab}=-\Xi_{ak}[\gamma_{b}](z)d(\gamma_{b}z),
\end{align*}
where $\Xi_{a\ell}$ is the canonical cocycle basis of \eqref{eq:xinorm} for $N=2$. 
Thus we find that 
\begin{align}\label{eq:delgamz}
	\partial_{a,\ell}(\gamma z)=-\Xi_{a\ell}[\gamma ](z)d(\gamma z),
	\end{align}
for all $\gamma\in\Gamma$. 
We have therefore established a natural pairing of $\partial_{a,\ell}$ with $\Xi_{a\ell}$. In conjunction with Proposition~\ref{prop:BersMapalpha} this implies
\begin{proposition}\label{prop:TCHm1}
 $\mathcal{F}_{2}\simeq \Eitwo\simeq T(\Cg)$  as vector spaces. 
\end{proposition}
By Lemma~\ref{lem:dimcobound} the  coboundary space $\Boundtwo\simeq \HOtwo$.
$\HOtwo$ is trivially isomorphic as a vector space to the M\"obius $\slLie_{2}(\C)$  Lie algebra where   $P=p(z)dz^{-1}\in \HOtwo$ for $p\in\Pi_{2}$ is  identified with $-p(z)\partial_{z}$.
By \eqref{eq:delgamz}, the $\Boundtwo$  element 
$\Xi_{P}=\sum_{\ell=0}^{2}\sum_{a\in\Ip}p_{a\ell}\Xi_{a\ell}$ of \eqref{eq:cocyclerels} is paired with $\mathcal{L}_{P}^{\Cg}\in T(\Cg)$ given by
\begin{align}
\mathcal{L}_{P}^{\Cg}:&=-\sum_{\ell=0}^{2}\sum_{a\in\Ip}p_{a\ell}\partial_{a,\ell}
=-\sum_{a\in\I}p(W_{a})\partial_{W_{a}},\label{eq:LP}
\end{align}
for the  original Schottky parameters $W_{\pm a}$. Thus $\{\mathcal{L}_{P}^{\Cg}\}$ generates the $\slLie_{2}(\C)$ subalgebra of $T(\Cg)$ associated with the M\"obius action \eqref{eq:Mobwrhoa}.
%%
%%\Xi_{P_{r}}=\sum_{\ell=0}^{2}\sum_{a\in\Ip}p_{r}(a,\ell)\Xi_{a\ell}$ 
%%in terms of the canonical basis for some complex numbers $p_{r}(a,\ell)$.
%which by \eqref{eq:delgamz} is given by
%%$\mathcal{L}_{r}=-\sum_{\ell=0}^{2}\sum_{a\in\Ip}p_{r}(a,\ell)\partial_{a,\ell}$. 
%\begin{align*}
%\mathcal{L}_{-1} &=-\sum_{a\in\Ip}\left(\partial_{a,0}+\rho_{a}^{-1}\partial_{a,2}\right)
%=-\sum_{a\in\I}\partial_{w_{a}},
%\\
%\mathcal{L}_{0} &=-\sum_{a\in\Ip}\left(w_{a}\partial_{a,0}+2\partial_{a,1}+w_{-a}\rho_{a}^{-1}\partial_{a,2}\right)
%=-\sum_{a\in\I}\left(w_{a}\partial_{w_{a}}+\rho_{a}\partial_{\rho_{a}}\right),
%\\
%\mathcal{L}_{1} &=-\sum_{a\in\Ip}
%\left(\left( w_{a}^2+\rho_{a}\right)\partial_{a,0}
%+2\left(w_{a}+w_{-a}\right)\partial_{a,1}
%+\left( 1+w_{-a}^2\rho_{a}^{-1}\right)\partial_{a,2}\right)
%\\
%&=-\sum_{a\in\I}\left((w_{a}^{2}+\rho_{a})\partial_{w_{a}}+2w_{a}\rho_{a}\partial_{\rho_{a}}\right).
%\end{align*}
%In terms of the original Schottky parameters $W_{\pm a}$ we find
%that $\mathcal{L}_{r}=-\sum_{a\in\I}W_{a}^{r+1}\partial_{W_{a}}$. 
%
In summary, we have the following vector space isomorphisms:
\begin{align}
\label{eq:cobound2}
\HOtwo \simeq   \Boundtwo \simeq  \slLie_{2}(\C)\subset T(\Cg).
\end{align}
Recalling that $\Schg=\Cg/\SL_{2}(\C)$ with tangent space $T(\Schg)=T(\Mg)$ we may consider the relevant quotients using Proposition~\ref{prop:TCHm1} and \eqref{eq:cobound2} to find
\begin{proposition}
	\label{prop:IsVS}
$  \Hgtwo\simeq \COHtwo \simeq T(\Mg)$ as vector spaces.
\end{proposition}
\medskip

We now discuss  the geometrical meaning of a given $N=2$ GEM form $\Psi_{2}(x,y)$. 
 Following \eqref{eq:GEMgamy} we define $\Theta_{a}(x;\ell)\in\Hgtwo$ for $a\in\Ip$ and $\ell=0,1,2$ by 
\begin{align*}
\Psi_{2}(x,y)-\Psi_{2}(x,\gamma_{a} y)=\sum_{\ell=0}^{2}\Theta_{a}(x,\ell)  (y-w_{a})^{\ell}dy^{-1},
\end{align*}
for each $\Gamma$ generator $\gamma_{a}$. It follows that for all $\gamma\in \Gamma$ 
\begin{align}
\Psi_{2}(x,\gamma y)-\Psi_{2}(x,y)=-\sum_{\ell=0}^{2}\sum_{a\in\Ip} \Theta_{a}(x,\ell)  \Xi_{a\ell}[\gamma](y),
\label{eq:ThetaPsi}
\end{align}
for all $\gamma\in \Gamma$. Hence $\{\Theta_{a}(x;\ell)\}$ is a $\Hgtwo$ spanning set by \eqref{eq:chigam}. 
We also define a canonical  differential operator  given by \cite{GT1,TW}
\begin{align}
\label{eq:nablaCg}
\delCg(x):=\sum_{\ell=0}^{2}\sum_{a\in\Ip} \Theta_{a}(x,\ell) \partial_{a,\ell}.
\end{align}
$\delCg(x)$ is a holomorphic vector field on $\Cg$ with quadratic differential coefficients. 

We next examine the dependence of $\delCg(x)$ on the choice of GEM function $\Psi_{2}$. From \eqref{eq:GEMdef} it is sufficient to consider a new GEM function $\widecheck{\Psi}_{2}(x,y)=\Psi_{2}(x,y)-\Phi(x)P(y)$ for some $\Phi\in\Hgtwo$ and $P\in \HOtwo$. \eqref{eq:ThetaPsi} implies  
$\widecheck{\Theta}_{a}(x,\ell)=\Theta_{a}(x,\ell)+p_{a\ell}\Phi(x)$  
for $\Xi_{P}=\sum_{\ell=0}^{2}\sum_{a\in\Ip}p_{a\ell}\Xi_{a\ell}$ so that 
\begin{align*}
\delCgcheck(x)=\delCg(x)+\Phi(x)\mathcal{L}_{P}^{\Cg}.
\end{align*}
Therefore modulo $\slLie_{2}(\C)$ we find that $\delCg(x)$ determines a unique vector field $\delMg(x)$  independent of the choice of GEM function. Thus we may choose a basis of $3g-3$ tangent vectors $\{\partial_{a\ell}\}_{\J}\in T(\Mg)$, where $\J$ is a set of distinct $(a,\ell)$ values  as described in \S~\ref{sect:normalized basis}, with corresponding cohomology basis $\{\Xi_{a\ell}\}_{\J}$, $\HOtwo$ basis $\{\Phi _{a\ell}\}_{\J}$ and Petersson dual basis $\{\Phi^{\vee}_{a\ell}\}_{\J}$. Then \eqref{eq:CanPsiQPeriod} implies that
\begin{align}
\label{eq:delSchottky}
\delMg(x)=\sum_{(a,\ell)\in\J}\Phi^{\vee}_{a\ell}(x)\partial_{a,\ell}.
\end{align}
Furthermore, for any coordinates $\{m_{r}\} _{r=1}^{3g-3}$ on  moduli space $\Mg$  with $T(\Mg)$ basis $\{\partial_{m_{r}}\}_{r=1}^{3g-3}$ and corresponding  $\Hgtwo$-basis  $\{\Phi_{r}\} _{r=1}^{3g-3}$ with dual basis  $\{\Phi^{\vee}_{r}\} _{r=1}^{3g-3}$ then  
 \cite{O}
\begin{align}
\label{eq:nablaMg}
\delMg(x)=\sum_{r=1}^{3g-3}\Phi^{\vee}_{r}(x)\partial_{m_{r}}.
\end{align}

\medskip

Next let $\Sgn$ denote a Riemann surface with $n$  punctures  $y_{1},\ldots,y_{n}\in\Sg$ with parameter space $\Cgn:=\Cg\times(\Sg)^{n}$.
Define the differential operator \cite{GT1,TW,O} 
\begin{align}
\label{eq:nablaCgn}
\delCgn(x):=\delCg(x)+\sum_{k=1}^{n}\Psi_{2}(x,y_{k})dy_{k}\,\partial_{y_{k}}.
\end{align}
We note that \eqref{eq:delgamz} and \eqref{eq:ThetaPsi} imply that for all $\gamma\in\Gamma$ 
\begin{align*}
\delCg(x)(\gamma y)
=\left(\Psi_{2}(x,\gamma y)-\Psi_{2}(x,y)\right)d(\gamma y),
\end{align*}
which implies that for each puncture $y_{k}$
\begin{align}\label{eq:delCGngamy}
\delCgn(x)(\gamma y_{k})=\Psi_{2}(x,\gamma y_{k})d(\gamma y_{k}).
\end{align}
\begin{proposition}
	$\delCgn(x)$ is a holomorphic vector field on $\Cgn$ for $x\in \Sgn$. 
\end{proposition}
\begin{proof}
We first show that $\delCgn(x)$ is a vector field on $\Cgn$ i.e. $\delCgn(x)$ is invariant under $y_{k}\rightarrow \gamma y_{k}$  for  $\gamma\in\Gamma$. Let 
$
(\widetilde{w}_{\pm a},\widetilde{\rho}_{a},\widetilde{y}_{k})=(w_{\pm a},\rho_{a},\gamma y_{k})$. Then 
\[
\partial_{a,\ell}=\widetilde{\partial}_{a,\ell}+\sum_{k=1}^{n}\partial_{a,\ell}(\widetilde{y}_{k})\partial_{\widetilde{y}_{k}},\quad \partial_{y_{k}}=\partial_{y_{k}}(\widetilde{y}_{k})\partial_{\widetilde{y}_{k}}.
\]
From \eqref{eq:nablaCgn} and  \eqref{eq:delCGngamy} we obtain 
\begin{align*}
\delCgn(x)=\delCgtilde(x)+\sum_{k=1}^{n}\delCgn(x)(\widetilde{y}_{k})\partial_{\widetilde{y}_{k}}
=\delCgtilde(x)+\sum_{k=1}^{n}\Psi(x,\widetilde{y}_{k})d\widetilde{y}_{k}
=\delCgntilde(x).
\end{align*} 
Thus $\delCgn(x)$ is a  vector field on $\Cgn$.  Furthermore $\delCgn(x)$ is holomorphic  since $\Psi_{2}(x,y)$ and $\Theta_{a}(x,\ell) $ are holomorphic for $x\in\Sgn$.
\end{proof}
For a new GEM function $\widecheck{\Psi}_{2}(x,y)=\Psi_{2}(x,y)-\Phi(x)P(y)$ we  find that
\begin{align*}
\delCgncheck(x)=\delCgn(x)+\Phi(x)\mathcal{L}_{P}^{\Cgn},
\end{align*}
where by \eqref{eq:LP}
\[
\mathcal{L}_{P}^{\Cgn}:=\mathcal{L}_{P}^{\Cg}-\sum_{k=1}^{n}p(y_{k})\partial_{y_{k}}
=-\sum_{a\in\I}p(W_{a})\partial_{W_{a}}-\sum_{k=1}^{n}p(y_{k})\partial_{y_{k}}, 
\]
 i.e. $\mathcal{L}_{P}^{\Cgn}$ generates an $\slLie_{2}(\C)$ subalgebra of $T(\Cgn)$. Thus modulo $\slLie_{2}(\C)$ we find that $\mathcal{L}_{P}^{\Cgn}$ determines a unique 
vector field $\delMgn(x)$ independent of the choice of GEM function. Similarly to the unpunctured case, we may choose a canonical cohomology basis $\{\Phi _{a\ell}\}_{\J}$  with corresponding canonical GEM form $\GEM_{2}^{{\Can}}$ of Proposition~\ref{prop:Can_norm} so that 
\begin{proposition}
There exists a unique holomorphic vector field	on $\Schgn$ 
\[
\delMgn(x)=\delMg(x)+\sum_{k=1}^{n}\GEM_{2}^{{\Can}}(x,y_{k})dy_{k}\,\partial_{y_{k}}.\]
\end{proposition}
\begin{remark}
The operators $\delCg$ and $\delCgn$ appear in the  conformal Ward identities for a vertex operator algebra on a Riemann surface as described in ref.~\cite{GT1}  (for genus 2) and in ref.~\cite{TW} (for genus $g$).
This leads to partial differential equations giving the variation with respect to moduli of Riemann surface structures  such as the bidifferential of the second kind $\omega(x,y)$, the projective connection, the holomorphic 1-differentials $\nu_{a}$ and the period matrix $\Omega_{ab}$. Thus Rauch's formula \cite{Ra} is rederived and expressed as
\begin{align}
\tpi\delMg\Omega_{ab}=\nu_{a}(x)\nu_{b}(x),\quad  a,b\in\Ip.\label{eq:Rauch}
\end{align}
In ref.~\cite{O} the operator $\delMg(x)$ is defined and  the operator $\delMgn(x)$ is anticipated but the existence and a construction of the $\GEM_{2}^{{\Can}}(x,y)$ term is not given. Similar partial differential equations for Riemann surface structures  (including the prime form) are also discussed  in ref.~\cite{O}. 
\end{remark}

%For example, Rauch's formula \cite{Ra} expressing the variation of the period matrix over moduli can be written as \cite{O,TW}
%\begin{align}
%\label{eq:rauch}
%\tpi\delMg(x)\,\Omega_{ab}=\nu_{a}(x)\nu_{b}(x),\quad a,b=1,\ldots,g,
%\end{align} 
%for holomorphic 1-differentials $\{\nu_{a}(x)\}$. 

%Define the following canonical basis $\HON$ for $N=2$
%\begin{align}
%P_{n}=-z^{n+1} dz^{-1},\quad n=-1,0,1,
%\end{align}
% with coboundary basis of 
%\begin{align*}
%\Xi_{P_{-1}}[\gamma_{a}]&=\left(1+\rho_{a}^{-1}
%z_{a}^{2}\right)dz^{-1},
%\\
%\Xi_{P_{0}}[\gamma_{a}]&=\left(w_{a}
%+2z_{a}+w_{-a}\rho_{a}^{-1}z_{a}^{2}
%\right)dz^{-1},
%\\
%\Xi_{P_{1}}[\gamma_{a}]&=\left(w_{a}^{2}+\rho_{a}
%+
%2( w_{a}+w_{-a} ) z_{a}+ \left( 1+w_{-a}^2\rho_{a}^{-1}\right) z_{a}^{2}\right)dz^{-1},
%\end{align*}

\section{The Green's Function $G_{N}(x,y)$}\label{sec:Green's}
We consider the Green's function $G_{N}(x,y)$ for the anti-holomorphic part of the Poincar\'e metric compatible  connection \cite{Ma,McI,McIT} and its relationship to a GEM form $\GEM_{N}$.  We develop some novel properties for the Green's function and describe an explicit formula for the inverse Bers map $\alpha^{-1}$ employing the Green's function. 

There is a unique connection acting on $\mathfrak{H}^{(g)}_{N}$  compatible with the Poincar\'e metric and the complex structure defined by e.g. \cite{McIT}
\begin{align*}
\partial_{N} \oplus \overline{\partial}_{N}:  \mathfrak{H}^{(g)}_{N}\rightarrow  \mathfrak{H}^{(g)}_{N+1}\oplus \mathfrak{H}^{(g)}_{N,1},
\end{align*}
where for $\Theta (z)=\theta (z) dz^N\in \mathfrak{H}^{(g)}_{N}$ we have 
\begin{align*}
 \partial_{N} \theta (z)&=\rho(z)^N\dz \left(\rho(z)^{-N}\theta (z)\right),
\qquad 
 \overline{\partial}_{N} \theta (z)=\dzb \theta (z).
\end{align*}
Note that $\ker  \overline{\partial}_{N}=\HgN$.
Define the projection  ${\mathcal P}_{N}:{\mathfrak H}^{(g)}_{N}\to \HgN$ by
\begin{align*}
%\label{eq:Profj}
{\mathcal P}_{N}\Theta := \sum_{r=1}^{d_{N}}\langle \Theta , \Phi _{r}\rangle \Phi _{r}^{\vee},
\end{align*}
for $\Theta =\theta(z)dz^{N}\in{\mathfrak H}^{(g)}_{N}$ and any  $\HgN$-basis $\{\Phi _{r}=\phi _{r}(z)dz^{N}\}_{r=1}^{d_{N}}$ with Petersson dual basis $\{\Phi ^{\vee}_{r}=\phi ^{\vee}_{r}(z)dz^{N}\}_{r=1}^{d_{N}}$. It is also useful to define the projection kernel 
\begin{align}
\label{eq:PNker}
p_{N}(x,y):=
\sum_{r=1}^{d_{N}}\phi _{r}^{\vee}(x) \overline{\phi _{r}(y)} \rho^{1-N}(y),
\end{align}
so that ${\mathcal P}_{N}\theta (x)=\iint_{\D}p_{N}(x,y)\theta(y)  d^{2}y$.

For $N\ge 2$ and $g\ge 2$, we define the Green's function  for $\overline{\partial}_{N}$ to be a bidifferential $(N,1-N)$-form $G_{N}(x,y)=g_{N}(x,y)dx^{N}dy^{1-N}$ where the regular part defined by 
\begin{align}
\label{eq:GN1}
g_{N}^{R}(x,y):=g_{N}(x,y)-\frac{1}{x-y},
\end{align}
satisfies the following two conditions:

\begin{enumerate}[(I)]
	\item $g_{N}^{R}(x,y)$ is holomorphic in $x$, 
	\item $g_{N}^{R}(x,y)$ is not meromorphic in $y$ with $
		\dfrac{1}{\pi}\dyb\, g_{N}^{R}(x,y)
	=  p_{N}(x,y)$. 
\end{enumerate}

\begin{remark}\leavevmode
	\begin{enumerate}[(i)]
	\item
We may heuristically rewrite (II) as 
	\[
	\frac{1}{\pi}\dyb g_{N}(x,y)=-\delta(y-x)+p_{N}(x,y),
	\]
	for Dirac delta function $\pi\delta(y-x)=\dyb (y-x)^{-1}$ from \eqref{eq:fdelta}.
	This  is a defining property for the Green's function for $\overline{\partial}_{N}$ in the physics literature e.g. \cite{EO,Ma}.	
	
	\item 
	We may also write (II)  in a coordinate-free way (similarly to \eqref{eq:dPhiF}) where
	\begin{align}
	\label{eq:dKNPhi}
	\frac{1}{\tpi}d \left(G_{N}(x,\cdot)\Phi \right)
	=\Phi\,P_{N}(x,\cdot),
	\end{align}
	on $\D-\{x\}$ for all $\Phi \in\HgN$ with
	\[
	P_{N}(x,y)=
	\sum_{r=1}^{d_{N}}\Phi _{r}^{\vee}(x) \overline{\Phi _{r}(y)}  \PM(y)^{-N}\omega(y).
	\] 
	\end{enumerate}
\end{remark}
\begin{lemma}
	\label{lem:KNunique}
	The Green's function is unique.
\end{lemma}
\begin{proof}
	Suppose $G_{N}$ and $\widetilde{G}_{N}$ are Green's functions  for $\overline{\partial}_{N}$. Let  
	\[
	G_{N}(x,y)-\widetilde{G}_{N}(x,y)= h_{N}(x,y)dx^{N}dy^{1-N},
	\]
	for $h_{N}(x,y)=g^{R}_{N}(x,y)-\widetilde{g}^{R}_{N}(x,y)$.
(II) implies $ h_{N}(x,y)dy^{1-N}$ is a holomorphic form of weight $1-N<0$ in $y$ (for fixed $x$). Hence $ h_{N}=0$  by the Riemann-Roch theorem.
\end{proof} 
\begin{proposition}
	\label{prop:KNperp}
	$\langle G_{N}(\cdot,y),\Phi\rangle=0$ for all $\Phi\in\HgN$.
\end{proposition}
\begin{proof}
For a given $\HgN$-basis $\{\Phi_{r}\} _{r=1}^{d_{N}}$, let $\lambda_{s}(y)dy^{1-N} =\langle G_{N}(\cdot,y),\Phi_{s}\rangle$.
Equation \eqref{eq:fdelta} implies that for $x,y\in \D$
\begin{align*}
\frac{1}{\pi}\dyb\lambda_{s}(y)&=\frac{1}{\pi}\dyb\iint_{\D}g_{N}(x,y)\,
\overline{\phi_{s}(x)}\rho(x)^{1-N}\,d^{2}x
\\
&=-\overline{\phi_{s}(y)}\rho(y)^{1-N}+\frac{1}{\pi}\iint_{\D}\dyb g_{N}^{R}(x,y)\,
\overline{\phi_{s}(x)}\rho(x)^{1-N}\,d^{2}x.
\end{align*}
Then  \eqref{eq:PNker} and condition (II)  imply
\begin{align*}
\frac{1}{\pi}\dyb\lambda_{s}(y)= -\overline{\phi_{s}(y)}\rho(y)^{1-N}+\sum_{r=1}^{d_{N}}\langle \Phi _{r}^{\vee},\Phi_{s}\rangle \overline{\phi _{r}(y)} \rho(y)^{1-N}=0.
\end{align*}
Thus $ \lambda_{s}(y)dy^{1-N}$ is a negative weight  holomorphic  form and hence
$ \lambda_{s}=0$.
%%
%Conversely, if 	$ \Lambda_{s}=0$ for $s=1,\ldots ,d_{N}$, we have  
%\begin{align}
%\label{eq:dbarkeqn}
%\frac{1}{\pi}\iint_{\D}\dyb g_{N}^{R}(x,y)\,
%\overline{\phi_{s}(x)}\rho(x)^{1-N}d^{2}x=\overline{\phi_{s}(y)}\rho^{1-N}(y).
%\end{align}
%Thus for fixed $x$ we find $\dyb g_{N}^{R}(x,y)dy^{1-N}d\ybar\in \mathfrak{H}^{(g)}_{1-N,1}$. 
%Let $X=\gamma x,Y=\gamma y$ and  $X'=(\gamma x)',Y'=(\gamma y)'$ for a $\gamma\in \Gamma$. Thus we find
%\begin{align*}
%\partial_{\overline y} k^{R}_{N}(X,y) X'^{N}-\dyb k^{R}_{N}(x,y)&=\partial_{\overline Y} k^{R}_{N}(X,Y) X'^{N}Y'^{1-N}\overline {Y'}-\dyb k^{R}_{N}(x,y)
%\\
%&=\dyb\left(\frac{1}{x-y}-\frac{X'^{N}Y'^{1-N}}{X-Y}\right)=0,
%\end{align*}
%since $G_{N}$ is a bidifferential $(N,1-N)$-form and using \eqref{eq:GN1}.
%Hence for fixed $y$ we find $\dyb g_{N}^{R}(x,y)dx^{N}\in \HgN$ since $g_{N}^{R}(x,y)$ is holomorphic in $x$. Thus 
%\[
%\frac{1}{\pi}\dyb g_{N}^{R}(x,y)=\sum_{r=1}^{d_{N}}c_{r}(y)\phi_{r}^{\vee}(x),
%\]
%for  coefficient functions $c_{r}(y)$ determined by \eqref{eq:dbarkeqn} to be
%$
%c_{r}(y)=\overline{\phi_{r}(y)}\rho^{1-N}(y).
%$
%Thus II holds. Hence (II'') is equivalent to II.
\end{proof}

We now construct  the unique Green's function \cite{Ma,McIT}.
 Let $\{\Phi _{s}\} _{s=1}^{d_{N}}$  be a $\HgN$-basis with potentials $\{{F} _{s}(y)\}_{s=1}^{d_{N}}$, Petersson dual basis  $\{\Phi ^{\vee}_{r}\} _{r=1}^{d_{N}}$ and GEM form $\GEM_{N}$.
 \begin{proposition} 
 	The Green's function $G_{N}(x,y)$ for $x,y\in\D$ is given by
 	\begin{align}
 	\label{eq:KNdefn}
 	G_{N}(x,y)=\GEM_{N}(x,y)+\sum_{r=1}^{d_{N}}\Phi ^{\vee}_{r}(x){F}_{r}(y).
 	\end{align}
 \end{proposition}
\begin{proof}
	$ \GEM_{N}(x,y)+\sum_{r}\Phi ^{\vee}_{r}(x){F}_{r}(y)$ is an $N$-differential with respect to $x$ and using \eqref{eq:xiF} and \eqref{eq:chigam}, it is a $1-N$ differential with respect to $y$. Thus condition I of  \eqref{eq:GN1} is verified.
It is straightforward to confirm condition II  using the Bers equation \eqref{eq:dF} and  that 
\[
\gem_{N}^{R}(x,y)=\gem_{N}(x,y)-\frac{1}{x-y},
\]
 is holomorphic in $y$. 
%\begin{align*}
%\dfrac{1}{\pi}\dyb \left( \widehat{k}_{N}^{R}(x,y)+\sum_{r=1}^{d_{N}}\phi ^{\vee}_{r}(x)f^{\Bers}_{r}(y)\right) =0+\sum_{r=1}^{d_{N}}\phi ^{\vee}_{r}(x)\overline{\phi_{r} (y)}\,\rho(y)^{1-N}
%\end{align*} 
\end{proof}
\begin{remark}
	\label{rem:GEM}
We note that Proposition~\ref{prop:KNperp} is easily verified using \eqref{eq:Bers_pot}. Furthermore, we can interpret  the Green's function as being the orthogonal projection of the GEM form $\GEM_{N}(x,y)$. Conversely, $\GEM_{N}(x,y)$ can be interpreted as a meromorphic extension of the unique  Green's function $G_{N}(x,y)$ which is why we refer to  $\GEM_{N}(x,y)$ as a Green's function with Extended Meromorphicity or GEM form. 
\end{remark}

The inverse map $\alpha^{-1}$  (associated with the inverse Bers map $\beta^{-1}$) which exists by Proposition~\eqref{prop:BersMapalpha} can  be explicitly described in terms of the  Green's function as follows:
\begin{proposition}\label{prop:GNF}
Let $\Xi$ be a 1-cocycle. Then $\Phi =\beta^{-1}(\Xi)$ has Bers potential \\ $F_{\Xi}(y)=\alpha^{-1}(\Xi)$ given by
\begin{align}
\label{eq:GNF}
F_{\Xi}(y)
=
\frac{1}{\tpi}\sum_{a=1}^{g}\oint_{\Con{a}}G_{N}(\cdot,y)\Xi[\gamma_{a}]
-\Xi[\lambda](y),	
\end{align}
where $\lambda\in\Gamma$ is the unique Schottky group element such that $\lambda y\in\D$. 
\end{proposition}
\begin{proof} 
 Let $\{\Phi _{s}\} _{s=1}^{d_{N}}$  be a $\HgN$-basis with potentials $\{{F} _{s}(y)\}$ and corresponding cocycle basis $\{\Xi_{s}\}$, Petersson dual basis  $\{\Phi ^{\vee}_{r}\}$ and GEM form $\GEM_{N}$. Write $\Xi=\sum_{r=1}^{d_{N}}x_{r}\Xi_{r}$ where 
 \[
 x_{r}=\frac{1}{\tpi} \sum_{a=1}^{g}\oint_{\Con{a}}\Phi ^{\vee}_{r}\Xi[\gamma_{a}],
 \]
 from Corollary~\ref{cor:PhiCont}. Then  \eqref{eq:KNdefn} and Proposition~\ref{prop:GEMCont}  imply
\begin{align*}
\frac{1}{\tpi}\sum_{a=1}^{g}\oint_{\Con{a}}G_{N}(\cdot,y)\Xi[\gamma_{a}]=&\frac{1}{\tpi} \sum_{a=1}^{g}\oint_{\Con{a}}\GEM_{N}(\cdot,y)\Xi[\gamma_{a}]+\sum_{r=1}^{d_{N}}x_{r}F _{r}(y)
\\
=&\,\Xi[\lambda](y)+F_{\Xi}(y),
\end{align*}
where $F_{\Xi}(y)=\sum_{r=1}^{d_{N}}x_{r}F _{r}(y)$.
\end{proof}
\begin{remark}\leavevmode
\label{rem:KNF}
\begin{enumerate}[(i)]
	\item 
	Define $F_{D}(y)=\frac{1}{\tpi}\sum_{a=1}^{g}\oint_{\Con{a}}G_{N}(\cdot,y)\Xi[\gamma_{a}]$
	for $y\in\Omo $. It follows that 
	$F_{\Xi}(y)=F_{D}(y)$ for all $y\in\D$ since $\Xi[\id]=0$. 
	
	\item 
	Note that $F_{D}|\gamma =F_{D}$ for all $\gamma\in\Gamma$ since $G_{N}(\cdot,y)$ is a $1-N$  form in $y$. Furthermore, \eqref{eq:GNF} implies that for all $\gamma\in\Gamma$ and $y\in\Omo$ 
	\begin{align*}
	F_{\Xi}(\gamma y)
	=&F_{D}(y)
	-\Xi[\lambda\gamma^{-1}](\gamma y).
	\end{align*}
	since $\lambda\gamma^{-1}(\gamma y)\in\D$. Thus 
	$	F_{\Xi}|_{\gamma}-F_{\Xi} =-\Xi[\lambda\gamma^{-1}]|_{\gamma}+\Xi[\lambda]
	=\Xi[\gamma]$ using the 1-cocycle condition \eqref{eq:cocycle}.
	Lastly, \eqref{eq:GNF}  implies that $F_{\Xi}$ satisfies   \eqref{eq:dPhiF} using   \eqref{eq:dKNPhi}.
	\item Proposition~\ref{prop:GNF} can be adapted to the case of any Kleinian group with contour integrals over appropriate boundaries of the  disconnected components of the fundamental domain $\D$ as in Remark~\ref{rem:Bersmap}.
	\item 	\eqref{eq:GNF} provides the motivation for a further generalization of the notion of a GEM form associated with a non-Kleinian uniformization of a genus two Riemann surface formed by sewing   two tori \cite{GT2}.  
\end{enumerate}
\end{remark}

\end{document}